\documentclass[12pt]{article}
\usepackage{graphicx,epsfig}
\usepackage{color}
\usepackage{amsmath,amssymb,mathrsfs,amsthm}
\usepackage{amsfonts}
\usepackage{multicol}
 \usepackage{graphicx}
\usepackage{color}
\usepackage{pdfcolmk}
\usepackage[nice]{nicefrac}
\usepackage{amsthm}
\usepackage{latexsym}
\usepackage{dsfont}
\usepackage{setspace}
\usepackage[a4paper, tmargin=1.5in, bmargin=1in, lmargin=1in, rmargin=1in, headheight=13.6pt]{geometry}
\usepackage[sort&compress,numbers]{natbib}
\usepackage[font=small,labelfont=bf,labelsep=space]{caption}
\makeatletter

\date{}

\def\@citex[#1]#2{\if@filesw\immediate\write\@auxout{\string\citation{#2}}\fi
  \def\@citea{}\@cite{\@for\@citeb:=#2\do
    {\@citea\def\@citea{,\linebreak[0]\hskip0pt plus .2em}%
      \@ifundefined{b@\@citeb}%
    {{\bf ?}\@warning{Citation `\@citeb' on page \thepage\space undefined}}%
      \hbox{\csname b@\@citeb\endcsname}}}{#1}}

\newtheorem{theorem}{Theorem}[section]

\newtheorem{problem}{Problem}[section]
\newtheorem{rule-def}[theorem]{Rule}
\numberwithin{equation}{section}
\makeatletter

\begin{document}
\title{Analytical method and its convergence analysis based on homotopy analysis  for the integral form of  doubly singular boundary value problems}
\author{Randhir Singh \thanks{Corresponding author. E-mail:randhir.math@gmail.com}\\
$\rm $\small {Department of Mathematics}\\
\small {Birla Institute of Technology Mesra,
 Ranchi -835215, India}} \maketitle{}
\begin{abstract}
\noindent In this paper, we consider the nonlinear doubly singular  boundary value problems $(p(x)y'(x))'+ q(x)f(x,y(x))=0,~0<x<1$ with Dirichlet/Neumann boundary conditions at $x=0$ and Robin type boundary conditions at $x=1$. Due to the presence of singularity at $x=0$ as well as discontinuity of $q(x)$ at $x=0$, these problems pose difficulties in obtaining their solutions. In this paper, a new formulation of the singular  boundary value problems is presented. To overcome the singular behavior at the origin, with the help of Green's function theory the problem  is transformed into an equivalent Fredholm integral equation. Then  the optimal homotopy analysis method is applied to solve integral form of problem. The optimal control-convergence parameter involved in the components of the series solution is obtained by minimizing the squared residual error equation. For speed up the calculations, the discrete averaged residual error is used to obtain optimal value of the  adjustable parameter $c_0$ to control the convergence of solution. The proposed method  \textbf{(a)} avoids solving a sequence of transcendental equations for the undetermined coefficients \textbf{(b)} it is a general method \textbf{(c)} contains a parameter $c_0$ to control the convergence of solution. Convergence analysis and  error estimate of the proposed method are discussed. Accuracy, applicability and generality of the present method is examined by solving five singular problems.

\end{abstract}

\noindent\textbf{Mathematics Subject Classification (AMC)}:  34B05,   34B15, 34B16, 	34B18, 	 34B27.
\section{Introduction}
We consider the nonlinear doubly singular  boundary value problems (DSBVPs) \cite{Bobisud1990,singh2014approximate,singh2014adomian}
\begin{eqnarray}\label{sec1:eq1}
-(p(x)y'(x))'= q(x)f(x,y(x)),~~~~~~~~0<x<1,
\end{eqnarray}
\begin{align}\label{sec1:eq2}
 y(0)=\delta_1,~~~\alpha_1\; y(1)+\beta_1\; y'(1)=\gamma_1,
\end{align}
or
\begin{align}\label{sec1:eq3}
\displaystyle\lim_{x\rightarrow0}p(x)y'(x)=0~ (\hbox{or}~ y'(0)=0),~~~\alpha_2\; y(1)+\beta_2\; y'(1)=\gamma_2,
\end{align}
where $\delta_1$,  $\alpha_i$, $\beta_i$ and $\gamma_i$ $i=1,2$ are any real constants.
Here, $p(0)=0$ and $q(x)$ may be discontinuous at $x=0$. Throughout this paper, the following conditions are assumed on $p(x)$, $q(x)$ and $f(x,y)$:
\begin{description}
\item $(C_1)$ $p(x)\in C[0,1]\cap C^{1}(0,1]$, $p(x)>0,~q(x)>0\in(0,1]$,
\item $(C_2)$ $ \displaystyle \frac{1}{p(x)}\in L^{1}(0,1]$ and $ \displaystyle \int\limits_{0}^{1}\frac{1}{p(x)}\int\limits_{x}^{1}q(s)ds\; dx<\infty$,  (BCs \eqref{sec1:eq2})
\item $(C_3)$   $q(x)\in L^{1}(0,1]$ and
$\displaystyle  \int\limits_{0}^{1}\frac{1}{p(x)}\int\limits_{0}^{x}q(s) ds\; dx<\infty$,~~ (for BCs \eqref{sec1:eq3})
 \item $(C_4)$ $f(x,y),~f_{y}(x,y)\in C(\Omega)$ and $f_{y}(x,y)\geq 0$ on  $ \Omega$, where $ \Omega:=\{(0,1]\times \mathbb{R}\}$.
\end{description}
 Equation \eqref{sec1:eq1} with  $p(x)=1$,  $q(x)=x^{-\frac{1}{2}}$
 and $f=y^{\frac{3}{2}}$ is  known as the Thomas-Fermi equations \cite{thomas1927calculation,fermi1927metodo}. The Lane-Emden  is a special case of  \eqref{sec1:eq1} with  $p(x)=q(x)=x^{\alpha}$ which has been used to model several phenomena in mathematical physics and astrophysics such as the theory of stellar structure, the thermal behavior of a spherical cloud of gas, isothermal gas spheres and the theory of thermionic currents \cite{wazwaz2017solving}.  The Lane-Emden equation is a basic equation in the theory of stellar structure for a shape factor of $\alpha=2$, i.e.,  signifying spherical bodies. Equation \eqref{sec1:eq1} with $p(x)=q(x)=x^{2}$  arises in oxygen diffusion in a spherical cell \cite{lin1976oxygen,anderson1980complementary} with  $f= \nicefrac{n y}{y+k},~n>0,~k>0,$
and in modelling of heat conduction in human head \cite{flesch1975distribution,gray1980distribution,duggan1986pointwise}
 with $f=\delta e^{-\theta y},~\theta>0,~\delta>0.$
 Existence and uniqueness of doubly singular  boundary value problems \eqref{sec1:eq1} with BCs.  \eqref{sec1:eq2} and \eqref{sec1:eq3} can be found in  \cite{chawla1987existence,dunninger1986existence,Bobisud1990,pandey2009note}. In general, such singular problems are difficult to solve due its singular behavior at $x=0$. There are several techniques to solve doubly singular  boundary value problems \eqref{sec1:eq1} with BCs. \eqref{sec1:eq3} where $p(x)=q(x)=x^{\alpha}$ for $\alpha>0$. The numerical study  of doubly singular  boundary value problems  has been carried out for past couple of decades and still it is an active area of research to develop some better numerical schemes. So far various numerical
methods such as the collocation methods \cite{Reddien1973projection,russell1975numerical}, tangent chord method   \cite{duggan1986pointwise}, finite difference methods \cite{jamet1970convergence,chawla1982finite,chawla1984finite}, spline finite difference methods \cite{iyengar1986spline}, B-Spline method  \cite{kadalbajoo2007b}, spline method \cite{kumar2007higher},  Chebyshev economization method \cite{kanth2003numerical}, Cubic spline method   \cite{kanth2005cubic,kanth2006cubic,kanth2007cubic},  Adomian decomposition method (ADM) and modified ADM \citep{inc2005different,mittal2008solution,khuri2010novel,ebaid2011new,Kumar2010}, ADM with Green's function  \cite{singh2013numerical,singh2014efficient,singh2016efficient}, variational iteration method (VIM) \cite{wazwaz2011comparison,wazwaz2011variational,ravi2010he}, the optimal modified VIM  \cite{singh2017optimal}, homotopy analysis method  \cite{danish2012note} and  homotopy perturbation method \cite{roul2016new} and the references cited therein. Solving   \eqref{sec1:eq1} using ADM or HAM  is always a computationally involved task as it requires computation of unknown coefficients.  In \cite{singh2013numerical,singh2014efficient,singh2016efficient},  the ADMGF was proposed  to overcome the difficulties occurred in the ADM    for  \eqref{sec1:eq1}. However, this method does not provide a mechanism to adjust and control the convergence region and rate of the series solutions.

In this paper, we use the OHAM to obtain approximate solutions of doubly singular problems  \eqref{sec1:eq1}. To overcome the singular behavior at the origin, the singular equation is transformed into an equivalent Fredholm integral equation and  then  the OHAM is applied to get approximate solutions. The most significant feature of the OHAM is the optimal control of the convergence of solutions by a convergence-control parameter  which ensures a very fast convergence. In summary, the OHAM has the following advantages:\vspace{-0.25cm}
\begin{itemize}
\item Unlike HAM, the present approach  does not require any additional computational work for unknown constants;\vspace{-0.2cm}
  \item Independent of small or large physical parameters;\vspace{-0.2cm}
  \item Guarantee of convergence;\vspace{-0.2cm}
  \item Flexibility on choice of base function and initial guess of solution;\vspace{-0.2cm}
  \item Useful analytic tool to investigate highly nonlinear problems with multiple solutions, singularity and perturbed.
\end{itemize}

\section{Description of the method}

\subsection{The equivalent integral form of  \eqref{sec1:eq1} and \eqref{sec1:eq2}}
Let us consider  the  homogeneous version of the problem \eqref{sec1:eq1} with \eqref{sec1:eq2} as
\begin{align}\label{sec2:eq1}
\left.
  \begin{array}{ll}
-(p(x)g(x))'=0,~~~x\in(0,1),\\
g(0)=\delta_1,~~\alpha_1\; g(1)+\beta_1 \;g'(1)=\gamma_1.
\end{array}
\right\}
\end{align}
Its solution is given by
\begin{align}\label{sec2:eq2}
g(x)=\delta_1+\frac{(\gamma_1-\delta_1\alpha_1)}{\mu}h(x),
\end{align}
where $$\mu=\alpha_1h(1)+\beta_1 h'(1),~~~~ h(x)=\displaystyle \int\limits_{0}^{x}\frac{dx}{p(x)},~~
~~ h(1)=\displaystyle\int\limits_{0}^{1}\frac{dx}{p(x)}~~ \hbox{and}~ ~h'(1)=\frac{1}{p(1)}.$$
Integrating  \eqref{sec1:eq1} twice w.r.t $x$ first from $x$ to 1 and then from 0 to $x$ and changing the order of integration, and applying the BCs $ y(0)=0,~\alpha_1 y(1)+\beta_1 y'(1)=0$, we obtain
\begin{align*}
y(x)=-\frac{1}{\mu}\int \limits_{0}^{1}\alpha_1h(x) h(s)q(s)f(s, y(s)) ds+\int \limits_{0}^{x}h(s)q(s)f(s, y(s))ds+\int \limits_{x}^{1}h(x) q(s)f(s, y(s))ds.
\end{align*}
Splitting the first integral into two parts from $0$ to $x$ and $x$ to $1$, we get
\begin{align*}
y(x)&=-\frac{1}{\mu}\int \limits_{0}^{x}\alpha_1h(x) h(s)q(s)f(s, y(s)) ds-\frac{1}{\mu}\int \limits_{x}^{1}\alpha_1h(x) h(s)q(s)f(s, y(s)) ds\\&~~~~+\int \limits_{0}^{x}h(s)q(s)f(s, y(s))ds+\int \limits_{x}^{1}h(x) q(s)f(s, y(s))ds.
\end{align*}
Combining the first and last, and  second and third integrals, we obtain
\begin{align}\label{sec2:eq3}
y(x)=\int\limits_{0}^{x} h(s)\left[1-\frac{\alpha_1h(x)}{\mu}\right] q(s)f(s, y(s))ds+\int\limits_{x}^{1}  h(x)\left[1-\frac{\alpha_1h(s)}{\mu}\right] q(s)f(s, y(s))ds.
\end{align}
Combining  \eqref{sec2:eq2} and   \eqref{sec2:eq3}, we  get  Fredholm integral form of doubly singular boundary value problems  \eqref{sec1:eq1}  and \eqref{sec1:eq2} as
\begin{align}\label{sec2:eq5}
y(x)=g(x) +\int\limits_{0}^{1}G(x,s)q(s) f(s,y(s))ds,
\end{align}
where $g(x)$ and  $G(x,s)$ are given by
 \begin{align}\label{sec2:eq6}
&g(x)=\delta_1+\frac{1}{\mu}(\gamma_1-\delta_1\alpha_1)h(x), \vspace{.2cm}\\
&G(x,s)=\left   \{
  \begin{array}{ll}
 h(x)\bigg[1-\frac{\alpha_1h(s)}{\mu}\bigg], & \hbox{$0\leq x\leq s $},\vspace{.2cm} \\
   h(s)\bigg[1-\frac{\alpha_1h(x)}{\mu}\bigg], & \hbox{$s\leq x\leq 1$.}
\end{array}
\right.
\end{align}
\subsection{The equivalent integral form of \eqref{sec1:eq1} and \eqref{sec1:eq3}}
We again consider the homogeneous  version of the problem \eqref{sec1:eq1} and \eqref{sec1:eq3}  as
\begin{align}\label{sec2:eq12}
\left.
  \begin{array}{ll}
-(p(x)g'(x))'=0,~x\in (0,1),\\
\displaystyle\lim_{x\rightarrow0+}p(x)g'(x)=0,~\alpha_2 g(1)+\beta_2 g'(1)=\gamma_2.
\end{array}
\right\}
\end{align}
The unique solution of \eqref{sec2:eq12} is  given by
\begin{align}\label{sec2:eq13}
g(x)=\frac{\gamma_2}{\alpha_2}.
\end{align}
Integrating  \eqref{sec1:eq1} w.r.to $x$  first from $0$ to $x$ and then from $x$ to $1$, then changing the order of integration, and applying the BCs $\lim_{x\rightarrow0+}p(x)y'(x)=0,~\alpha_2 y(1)+\beta_2 y'(1)=0$, we get
\begin{align*}
y(x)&=\int\limits_{0}^{1}\frac{\beta_2}{\alpha_2p(1)}q(s)f(\xi, y(s))ds+\int \limits_{0}^{1}\bigg[\int
\limits_{s}^{1}\frac{dx}{p(x)}\bigg]q(s)f(s, y(s))ds-\int \limits_{0}^{x}\bigg[\int \limits_{s}^{x}\frac{dx}{p(x)}\bigg]q(s)f(s, y(s))ds.
\end{align*}
Splitting the first  and second integrals into two parts from $0$ to $x$ and $x$ to $1$, we get
\begin{align*}
y(x)&=\int\limits_{0}^{x} \frac{\beta_2}{\alpha_2p(1)}q(s)f(s, y(s))ds+\int\limits_{x}^{1} \frac{\beta_2}{\alpha_2p(1)}q(s)f(s, y(s))ds+\int \limits_{0}^{x}\bigg[\int\limits_{s}^{1}\frac{dx}{p(x)}\bigg]q(s)f(s, y(s))ds\\&~~~~+\int \limits_{x}^{1}\bigg[\int\limits_{s}^{1}\frac{dx}{p(x)}\bigg]q(s)f(s, y(s))ds-\int \limits_{0}^{x}\bigg[\int \limits_{s}^{x}\frac{dx}{p(x)}\bigg]q(s)f(s, y(s))ds,~s>0.
\end{align*}
By combining the integrals of same limits, we obtain
\begin{align}\label{sec2:eq14}
y(x)&=\int \limits_{0}^{x}\bigg[\int
\limits_{s}^{1}\frac{dx}{p(x)}-\int\limits_{s}^{x}\frac{dx}{p(x)} +\frac{\beta_2}{\alpha_2p(1)} \bigg]q(s)f(s, y(s))ds+\int \limits_{x}^{1}\bigg[\int
\limits_{s}^{1}\frac{dx}{p(x)}+\frac{\beta_2}{\alpha_2p(1)} \bigg]q(s)f(s, y(s))ds.
\end{align}
Combining  \eqref{sec2:eq13} and  \eqref{sec2:eq14}, we  get  Fredholm integral form of doubly singular boundary value problem  \eqref{sec1:eq1}  and \eqref{sec1:eq3} as
\begin{align} \label{sec2:eq15}
y(x)=g(x)+\int\limits_{0}^{1} G(x,s)q(s) f(s,y(s))ds.
\end{align}
where $g(x)$ and  $G(x,s)$ are  given by
\begin{align}\label{sec2:eq16}
&g(x)=\frac{\gamma_2}{\alpha_2},\vspace{.2cm}\\
& G(x,s)=\left   \{
  \begin{array}{ll}
   \displaystyle \int\limits_{s}^{1}\frac{dx}{p(x)}+\frac{\beta_2}{\alpha_2\; p(1)}, & \hbox{$0< x \leq s$},\vspace{.2cm} \\
    \displaystyle \int\limits_{s}^{1}\frac{dx}{p(x)}-\int\limits_{s}^{x}\frac{dx}{p(x)} +\frac{\beta_2}{\alpha_2p(1)}, & \hbox{$s\leq x\leq1$.}
\end{array}
\right.
\end{align}
or
\begin{align}
& G(x,s)=\left   \{
  \begin{array}{ll}
   \displaystyle h(1)-h(s)+\frac{\beta_2}{\alpha_2 }h'(1), & \hbox{$0< x \leq s$},\vspace{.2cm} \\
    \displaystyle h(1)-h(x) +\frac{\beta_2}{\alpha_2}h'(1), & \hbox{$s\leq x\leq1$.}
\end{array}
\right.
\end{align}

\subsection{Analytical method based on homotopy analysis}
The integral equations \eqref{sec2:eq5} or \eqref{sec2:eq15} may be
written in the operator equation form
\begin{align}\label{sec2:eq17}
\mathcal{T}[y(x)]=y(x)-g(x)-\int\limits_{0}^{1}G(x,s) q(s) f(s,y(s))ds=0,
\end{align}
where $g(x)$ and $G(x,s)$ are given by \eqref{sec2:eq6} or \eqref{sec2:eq16}, respectively. Basic idea of  homotopy analysis method for solving different scientific models
can be found in  \cite{liao1995approximate,liao2003beyond,liao2007general,liao2009series,abbasbandy2013determination,fan2013optimal} and optimal homotopy asymptotic method in \cite{marinca2008application,hericsanu2010accurate,hashmi2012numerical}. According to homotopy analysis method, using $r\in[0, 1]$ as an embedding parameter, the general zero-order deformation equation is constructed as
\begin{align}\label{sec2:eq18}
(1-q)[\phi(x;r)-y_0(x)]=r\; c_0\;  \mathcal{T}[\phi(x;r)],
\end{align}
where $y_0(x)$ denotes an initial guess for the exact solution $y(x)$,  $c_0\neq0$ is convergence-controller parameter,   $\phi(x;r)$ is an unknown function and $\mathcal{N}[\phi(x;r)]$ is given by
\begin{align}\label{sec2:eq19}
\mathcal{T}[\phi(x;r)]=\phi(x;r)-g(x)-\int\limits_{0}^{1}G(x,s) q(s) f(s,\phi(s;r))ds=0.
\end{align}
When $r = 0$, the zero-order deformation \eqref{sec2:eq18} becomes $\phi(x;0)=y_0(x),$ and
when $r = 1$, it leads to $\mathcal{T}[\phi(x;1)]=0,$  which is exactly the same as the original problem \eqref{sec2:eq17} provided that $\phi(x;1)= y(x)$. Expanding the function $\phi(x;r)$ in a Taylor series with respect to the parameter $r$, we obtain
\begin{align}\label{sec2:eq20}
\phi(x;r)=y_0(x)+\sum_{k=1}^{\infty} y_k(x) r^{k},
\end{align}
where $y_k(x)$ is given by
\begin{align}\label{sec2:eq21}
y_k(x)=\frac{1}{k!}\frac{\partial^k \phi(x;r)}{\partial r^k}\bigg|_{r=0}.
\end{align}
If the convergence controller parameter $c_0\neq0$ is
chosen properly, the series \eqref{sec2:eq20}  converges for $r = 1$ and it  becomes
\begin{align}\label{sec2:eq22}
\phi(x;1)\equiv y(x)=y_0(x)+\sum_{k=1}^{\infty} y_k(x),
\end{align}
which will be one of solutions of the problem \eqref{sec2:eq17}.

Defining the vector $\overrightarrow{y}_k = \{y_0 (x), y_1 (x), \ldots, y_k (x)\}$ and differentiating \eqref{sec2:eq18},  $k$ times with respect to the parameter $r$, dividing it by $k!$, setting subsequently $r = 0,$ we obtain the $k$th-order deformation equation as
\begin{align}\label{sec2:eq23}
y_k(x)-\chi_{k}\; y_{k-1}(x)=c_0\ R_k(\overrightarrow{y}_{k-1},x),
\end{align}
where $\chi_{k}$ is given by
\begin{align}\label{sec2:eq24}
\chi_{k}=\displaystyle\left   \{
  \begin{array}{ll}
     0, & \hbox{$k=0,1$} \\
    1, & \hbox{$k\geq2$}
\end{array}
\right.
\end{align}
and
\begin{align}\label{sec2:eq25}
\nonumber R_k(\overrightarrow{y}_{k-1},x)&=\frac{1}{(k-1)!} \bigg[ \frac{\partial^{k-1}} {\partial r^{k-1}} \mathcal{T}\bigg( \displaystyle \sum_{j=0}^{\infty} y_j r^{j}\bigg)\bigg]\bigg|_{r=0}\\
                  &=y_{k-1}(x)-(1-\chi_{k})g(x)-\int\limits_{0}^{1}G(x,s)\; q(s)\; \mathcal{D}_{k-1}[f(\phi)] \;ds
\end{align}
where $\mathcal{D}_{k-1)}[f(\phi)]$ is the $(k-1)$th-order homotopy-derivative operator \cite{liao2012homotopy} given by
\begin{align}\label{sec2:eq26}
\mathcal{D}_{k-1}[f(\phi)]=\frac{1}{(k-1)!}\frac{\partial^{k-1} }{\partial r^{k-1}} f\bigg(x, \displaystyle \sum_{j=0}^{\infty} y_j r^{j}\bigg)\bigg|_{r=0}.
\end{align}
Using \eqref{sec2:eq23} and  \eqref{sec2:eq25}, the $k$th-order deformation equation is simplified as
\begin{align}\label{sec2:eq27}
y_k(x)-\chi_{k} y_{k-1}(x)=c_0\ \bigg[y_{k-1}(x)-(1-\chi_{k})g(x)-\int\limits_{0}^{1}G(x,s) q(s)  \mathcal{D}_{k-1}[f(\phi)] ds\bigg].
\end{align}
Using  \eqref{sec2:eq27} with an initial guess $y_0(x)=g(x)$, the solution components $y_k(x)$ are obtained  as
\begin{align}\label{sec2:eq27a}
\left.
  \begin{array}{ll}
 y_1(x)&=\displaystyle c_0\ \bigg\{y_{0}(x)-g(x)-\int\limits_{0}^{1}G(x,s) q(s)  \mathcal{D}_{0}[f(\phi)] ds \bigg\}\\
  y_2(x)&=\displaystyle(1+c_0)\ y_{1}(x)-c_0\bigg\{\int\limits_{0}^{1}G(x,s) q(s)  \mathcal{D}_{1}[f(\phi)] ds\bigg\}\\
 \vdots\\
 y_k(x)&=\displaystyle(1+c_0)\ y_{k-1}(x)-c_0\bigg\{\int\limits_{0}^{1}G(x,s) q(s)  \mathcal{D}_{k-1}[f(\phi)] ds\bigg\}~~~ k\geq3
 \end{array}
\right\}
\end{align}
The $M$th-order approximate solution of the problem \eqref{sec2:eq17} is defined as
\begin{align}\label{sec2:eq28}
\phi_M(x,c_0)=y_0(x)+\sum_{k=1}^{M} y_k(x,c_0).
\end{align}
Appropriate selection of the convergence control parameter $c_0$ has a big influence
on the convergence region of series \eqref{sec2:eq22} and on the convergence rate as well \cite{liao2012homotopy,odibat2010study}. One of the methods for selecting the value of convergence control parameter is the so-called $c_0$-curve and  the horizontal line may be considered  as the
valid interval for $c_0$ \cite{liao2003beyond,hetmaniok2014usage}. This method enables to determine the effective region of the convergence control parameter,
however it does not give the possibility to determine the value ensuring the fastest convergence \cite{liao2012homotopy}.  Another way to find the optimal value of the convergence-control parameter $c_0$ is obtained by minimizing the squared residual of governing equation
\begin{align}\label{sec2:eq29}
E_{M}(c_0)=  \int_{0}^{1}  \bigg(\mathcal{T}[\phi_M(x,c_0)]\bigg)^2  dx.
\end{align}
The squared residual error defined by \eqref{sec2:eq29} is a kind of measurement of the accuracy of the $M$th-order
approximation. However, the exact squared residual error is expensive to calculate when $M$ is large. For speed up the calculations Liao \cite{liao2010optimal,liao2012homotopy} suggested to replace the integral in formula
\eqref{sec2:eq29} by its approximate value obtained by applying the quadrature rules. So, we
approximate $E_M$ by  the discrete averaged residual error defined by
\begin{align}\label{sec2:eq30}
E_{M}(c_0)\approx \frac{1}{n}\sum_{j=1}^{n} \bigg(\mathcal{T}[\phi_M(x_j,c_0)]\bigg)^2 ,
\end{align}
where $0=x_1<x_2<\ldots x_{j-1}<x_{j}<\ldots <x_n=1$ with nodal points $x_j = jh,$ $h=x_j-x_{j-1}$, ~ $j =1,2,\ldots,n$.  Since $E_{M}(c_0)$ dependent upon $c_0$, the optimal value is obtained by solving $\nicefrac{dE_{M}}{dc_0}=0$, the effective region of the convergence control parameter is usually defined as $R_{c_{0}}=\{c_0: \lim_{M\rightarrow \infty}E_{M}(c_0)=0\}$ and optimal value will satisfy $E_{M}(\hat{c}_{0})<E_{M}(c_0)$. Having computed the optimal value $\hat{c}_0$ and substituting in  \eqref{sec2:eq28}, the approximate solution will be obtained.

\section{Convergence analysis}
In this section, we establish the convergence of method defined in \eqref{sec2:eq27a} the
solution of equivalent integral form \eqref{sec2:eq17} of doubly singular boundary value problems  \eqref{sec1:eq1} -\eqref{sec1:eq3}. Let $\mathds{X}= \big(C[0,1], \|y\|\big)$ be a Banach space with $\|y\|=\max_{ x\in [0,1]} |y(x)|,~y\in \mathds{X}.$

\begin{theorem}\label{sec3:eq1}
Let $0 < \delta< 1$ and the solution components $y_0(x),y_1(x),y_2(x),\ldots$ obtained by \eqref{sec2:eq27a} satisfy the following condition:
 \begin{align}\label{sec3:eq2}
\exists~~ k_0 \in \mathbb{N}~~\forall~ k\geq k_0:~ \|y_{k+1}\|\leq \delta\|y_{k}\|,
 \end{align}
then the series solution $\sum_{k=0}^{\infty} y_k(x)$ is convergent.
\end{theorem}
\begin{proof}
Define the sequence $\{\phi_n\}_{n=0}^{\infty}$ as,
\begin{align}\label{sec3:eq3}
\left\{
  \begin{array}{ll}
\phi_0=y_0(x)\\
\phi_1=y_0(x)+y_1(x)\\
\phi_2=y_0(x)+y_1(x)+y_2(x)\\
\vdots\\
\phi_n=y_0(x)+y_1(x)+y_2(x)+ \cdots +y_n(x)\\
\end{array} \right.
\end{align}
and we show that is a Cauchy sequence in the Banach space  $\mathds{X}$. For this purpose, consider
\begin{align*}
\|\phi_{n+1}-\phi_{n}\|&=\|y_{n+1}\|\leq \delta \|y_{n}\|\leq  \delta^2 \|y_{n-1}\|\leq \ldots \leq \delta^{n-k_0+1} \|y_{k_{0}}\|.
\end{align*}
For every $n,m\in \mathbb{N}$, $n\geq m >k_0$, we have
\begin{align}\label{sec3:eq4}
\nonumber \|\phi_{n}-\phi_{m}\|&=\|(\phi_n-\phi_{n-1})+(\phi_{n-1}-\phi_{n-2})+\cdots+(\phi_{m+1}-\phi_{m})\|\\
\nonumber&\leq \|\phi_n-\phi_{n-1}\|+\|\phi_{n-1}-\phi_{n-2}\|+\cdots+\|\phi_{m+1}-\phi_{m}\|\\
\nonumber &\leq   (\delta^{n-k_0} + \delta^{n-k_0-1} +\cdots+\delta^{m-k_0+1})\|y_{k_{0}}\|\\
&= \frac{1-\delta^{n-m}}{1-\delta}\delta^{m-k_{0}+1}   \|y_{k_{0}}\|
\end{align}
and since $0<\delta<1$ so it follows that
\begin{eqnarray}\label{sec4:eq5}
\lim_{n,m\rightarrow \infty} \|\phi_{n}-\phi_{m}\|=0.
\end{eqnarray}
Therefore, $\{\phi_n\}_{n=0}^{\infty}$ is a Cauchy sequence in the Banach space  $\mathds{X}$ and it implies that the series solution defined in \eqref{sec2:eq28}
converges.  This completes the proof of Theorem \ref{sec3:eq1}.
\end{proof}

\begin{theorem}\label{sec3:eq6}
Assume that the series solution $\sum_{k=0}^{\infty} y_k(x)$ defined in  \eqref{sec2:eq28}, is convergent to the solution $y(x)$. If the truncated
series $\phi_M(x,c_0)=\sum_{m=0}^{M} y_m(x,c_0)$ is used as an approximation to the solution $y(x)$ of the problem \eqref{sec2:eq17}, then the maximum
absolute truncated error is estimated as
\begin{align}
\left|y(x)-\phi_M(x,c_0)\right|\leq\frac{1}{1-\delta}\delta^{M-k_{0}+1}   \|y_{k_{0}}\|.
\end{align}
\end{theorem}
\begin{proof}
From Theorem \ref{sec3:eq1}, following inequality \eqref{sec3:eq4}, we have
\begin{align*}
\|\phi_{n}-\phi_M(x,c_0)\|\leq\frac{1-\delta^{n-M}}{1-\delta}\delta^{M-k_{0}+1}   \|y_{k_{0}}\|,
\end{align*}
for $n\geq M$. Now, as $n\rightarrow \infty$ then   $\phi_{n}\rightarrow y$ and $\delta^{n-M}\rightarrow0$. So,
\begin{align}\label{sec4:eq7}
\|y(x)-\phi_M(x,c_0)\|\leq\frac{1}{1-\delta}\delta^{M-k_{0}+1}   \|y_{k_{0}}\|.
\end{align}
Theorems \ref{sec3:eq1} and \ref{sec3:eq6} together confirm that the convergence of series solution \eqref{sec2:eq28}.
\end{proof}

Now, we discuss about the uniqueness of the solution of  problem \eqref{sec2:eq17}. The operator equation  form of \eqref{sec2:eq17} is written as
\begin{align}\label{sec3:eq8}
y(x)=g(x)+\int\limits_{0}^{1} G(x,s)q(s)f(s,y(s))ds.
\end{align}
\begin{theorem}\label{sec3:eq9}
Let $0 <\delta< 1$ and suppose there exists a constant $l>0$ such that
\begin{align}\label{sec3:eq10}
|f(x,z_1)-f(x,z_2)|\leq L |z_1-z_2|~~ \forall~ (x,z_1), (x,z_2)\in \Omega.
\end{align}
Then there exists one, and only one, solution $y(x)$ of  equation \eqref{sec3:eq8} in $\mathds{X}$.
\end{theorem}
\begin{proof}
Let us assume that there are two solutions $z_1(x)$, $z_1(x)\in \mathds{X}$ of the problem \eqref{sec3:eq8}
 \begin{align*}
\|z_1-z_2\|&=\max_{ x \in [0,1]}|z_1(x)-z_2(x)|=\max_{ x \in [0,1]} \bigg|\int\limits_{0}^{1} G(x,s)q(s)\big( f(s,z_1(s))- f(s,z_2(s))\big) ds \bigg|\\
&\leq   \max_{ s \in [0,1]}\left| f(s,z_1(s))- f(s,z_2(s))\right| \bigg( \max_{ x \in [0,1]} \int\limits_{0}^{1} |G(x,s)q(s)ds|\bigg)
\end{align*}
using \eqref{sec3:eq10},  the above inequality reduce to
 \begin{align*}
\|z_1-z_2\|&\leq  LM \max_{ s\in [0,1]} |z_1(s)-z_2(s)|=\delta\|z_1-z_2\|,
\end{align*}
setting $\delta=LM$ and $M:=\max_{ x\in [0,1]}\int\limits_{0}^{1}|G(x,s)q(s)ds|$, we obtain
 \begin{align*}
\|z_1-z_2\|\leq \delta\|z_1-z_2\|,
\end{align*}
since $0<\delta<1$, the equality $z_1=z_2$ must hold. This means  equation \eqref{sec3:eq8} has a unique solution in $\mathds{X}$.
\end{proof}
In the following theorem we show that the series  defined in \eqref{sec2:eq28} is convergent, where the solution components  $y_k(x)$ are obtained from \eqref{sec2:eq27a}, then it must be a solution of the integral of \eqref{sec3:eq8}.

\begin{theorem}\label{sec3:eq11}
Assume that the series solution $\sum_{k=0}^{\infty} y_k(x)$ defined in  \eqref{sec2:eq28} , is convergent to the solution $y(x)$
then it must be a solution of the integral of \eqref{sec3:eq8}.
\end{theorem}
\begin{proof}
Since  $\sum_{k=0}^{\infty} y_k(x)$  is convergent, then
\begin{align}\label{sec3:eq12}
\lim_{n\rightarrow\infty} y_n(x)=0,~~~~ \forall~~x\in[0,1].
\end{align}
By summing up the left hand-side of  \eqref{sec2:eq23}, we get
\begin{align}\label{sec3:eq13}
\sum_{k=1}^{n}[y_k(x)-\chi_{k} y_{k-1}(x)]=y_1(x)+\ldots+(y_n(x)-y_{n-1}(x))=y_n(x).
\end{align}
Letting $n\rightarrow\infty$ and using \eqref{sec3:eq12}, equation \eqref{sec2:eq13} reduces to
\begin{align}\label{sec2:eq14}
\sum_{k=1}^{\infty}[y_k(x)-\chi_{k} y_{k-1}(x)]= \lim_{n\rightarrow\infty} y_n(x)=0.
\end{align}
Using \eqref{sec2:eq14}  and right hand-side of the relation   \eqref{sec2:eq23},  we obtain
\begin{align}\label{sec2:eq15}
\sum_{k=1}^{\infty} c_0\ R_k(\overrightarrow{y}_{k-1},x)=\sum_{k=1}^{\infty}[y_k(x)-\chi_{k} y_{k-1}(x)]=0.
\end{align}
Since $c_0\neq 0$, then equation \eqref{sec2:eq15} reduces to
\begin{align}\label{sec2:eq16}
\sum_{k=1}^{\infty} \ R_k(\overrightarrow{y}_{k-1},x)=0.
\end{align}
 Using \eqref{sec2:eq16} and \eqref{sec2:eq25}, we have
\begin{align*}
0=\sum_{k=1}^{\infty} \ R_k(\overrightarrow{y}_{k-1},x)&=\sum_{m=1}^{\infty} \bigg[y_{k-1}(x)-(1-\chi_{k})g(x)-\int\limits_{0}^{1}G(x,s)\;q(s)\;  \mathcal{D}_{k-1}[f(s,\phi)] ds\bigg]\\
 &=  \sum_{k=1}^{\infty} y_{k-1}(x)-g(x)-\int\limits_{0}^{1}G(x,s)\; q(s)\;  \sum_{k=1}^{\infty} \mathcal{D}_{k-1}[f(s,\phi)]  ds,
  \end{align*}
since $\sum_{k=0}^{\infty} y_k(x)$  converges to $y(x)$, then  $\sum_{k=0}^{\infty} \mathcal{D}_{k-1}[f(x,\phi)]$  converges to $f(x,y)$ \cite{cherruault1989convergence},
\begin{align*}
y(x)=g(x)+\int\limits_{0}^{1}G(x,s) q(s)  f(s,y(s))  ds.
\end{align*}
Hence, $y(x)$ is the exact solution of  integral equation \eqref{sec3:eq8}.
\end{proof}

\section{Numerical results}

To examine the accuracy and applicability of the OHAM, we consider five  examples of singular boundary value problem.  All of the computations have been performed using MATHEMATICA.  
Here,  $y(x)$,  $\phi_{M}(x)$ and $\psi_{M}(x)= \phi_{M}(x,-1)$ denote   the exact, OHAM, and ADMGF solutions, respectively.

\begin{problem}\label{prob1}
\rm{\textbf{Doubly Singular Boundary Value Problem} \cite{singh2013numerical,singh2016efficient}}
\end{problem}
Consider nonlinear doubly  singular boundary value problems
\begin{eqnarray*}
\left.
  \begin{array}{ll}
    \displaystyle (x^{\alpha}y'(x))'=  x^{\alpha+\beta-2}[\beta \left(\beta  x^{\beta }e^{2y}-e^y(\alpha +\beta -1)\right)],~~~~0<x<1,\vspace{0.15cm} \\
 \displaystyle y(0)=\ln \bigg(\frac{1}{4}\bigg),~~~~~~y(1)=\ln \bigg(\frac{1}{5}\bigg),~~~0<\alpha<1,~~  \beta>0.
\end{array}
\right\}
\end{eqnarray*}
Here, $p(x)=x^{\alpha}$, $q(x)=x^{\alpha+\beta-2}$, $f(x,y)=[\beta \left(\beta  x^{\beta }e^{2y}-e^y(\alpha +\beta -1)\right)]$ with $\delta_1=\ln \big(\frac{1}{4}\big)$, $\alpha_1=1$, $\beta_1=0$ and $\gamma_1=\ln \big(\frac{1}{5}\big)$. Its exact solution is $y(x)= \ln\big(\frac{1}{4+x^{\beta}}\big)$.
Applying OHAM  \eqref{sec2:eq27} with an initial guess $y_0(x)=\ln \big(\frac{1}{4}\big)$,  we get the approximate solution $\phi_M(x,c_0).$
Using the formula \eqref{sec2:eq30},  we obtain optimal values, $\hat{c}_0=[-0.970001;-0.970011]$ (for $\alpha=0.5,~\beta=1$) with $M=5,~10$, respectively.
We define absolute error as
\begin{align*}
E_{a}^{M}=|y(x)-\phi_M(x)|~~~~\hbox{and}~~~e_M^{a}=|y(x)-\psi_{M}(x)|,~~x\in[0,1].
\end{align*}
The numerical results of absolute errors and approximate solutions  are shown in Table \ref{tab1}. One can see that OHAM method provides better results compared with  ADMGF method.
\begin{table}[htbp]
\caption{Results of absolute  error and solutions  of Problem \ref{prob1} when  $\alpha=0.5,~\beta=1$}\label{tab1}
\centering
\vspace{-0.3cm}
\renewcommand{\arraystretch}{1.1}
\setlength{\tabcolsep}{0.11in}
\begin{tabular}{l|cc| cc|cc}
\hline
\cline{1-7}
$x$ & $e_{a}^{M}$  & $E_{a}^{5}$ & $e_{a}^{10}$    &   $E_{a}^{10}$    &   $\psi_{10}$ & $\phi_{10}$\\
\cline{1-7}
0.1	&	3.16E-08	&	2.53E-08	&	5.38E-13	&	4.28E-14	&	-1.410986974	&	-1.410986974	\\
0.2	&	4.70E-08	&	3.44E-08	&	8.19E-13	&	5.24E-14	&	-1.435084525	&	-1.435084525	\\
0.3	&	6.15E-08	&	3.86E-08	&	1.07E-12	&	5.55E-14	&	-1.458615023	&	-1.458615023	\\
0.4	&	7.62E-08	&	3.94E-08	&	1.32E-12	&	5.70E-14	&	-1.481604541	&	-1.481604541	\\
0.5	&	9.11E-08	&	3.73E-08	&	1.54E-12	&	5.92E-14	&	-1.504077397	&	-1.504077397	\\
0.6	&	1.04E-07	&	3.28E-08	&	1.71E-12	&	6.17E-14	&	-1.526056303	&	-1.526056303	\\
0.7	&	1.12E-07	&	2.63E-08	&	1.74E-12	&	6.21E-14	&	-1.547562509	&	-1.547562509	\\
0.8	&	1.08E-07	&	1.81E-08	&	1.53E-12	&	5.24E-14	&	-1.568615918	&	-1.568615918	\\
0.9	&	7.72E-08	&	8.07E-09	&	9.51E-13	&	2.90E-14	&	-1.589235205	&	-1.589235205	\\

\hline
\end{tabular}
\end{table}

\begin{problem}\label{prob2}
\rm{\textbf{Thermal Explosions} \cite{fogler1999elements,danish2012note} }
\end{problem}
Consider nonlinear singular boundary value problems
\begin{eqnarray*}
\left.
  \begin{array}{ll}
\displaystyle (x^2y'(x))' =\sigma^2 x^2 y^{n}(x),~~~~~0<x<1,\vspace{0.15cm}\\
   \displaystyle y'(0)=0,~~~~y(1)=1,
\end{array}
\right\}
\end{eqnarray*}
 Here, $p(x)=q(x)=x^{2}$, $f(x,y)=\sigma^2 y^{n}(x)$ with  $\alpha_2=1$, $\beta_2=0$ and $\gamma_2=1$. Applying OHAM  \eqref{sec2:eq27} with an initial guess $y_0(x)=1$, we get the approximate solution $\phi_M(x,c_0).$ Using the  formula \eqref{sec2:eq30},  we obtain optimal values $\hat{c}_0=[-0.8929193;-0.8712345]$, (for $n=1.5$,~$ \sigma=1$),  $\hat{c}_0=[-0.6890655; -0.6666666]$ (for $n=2, \sigma=1.5$) and $\hat{c}_0=[-0.5723102; -0.4809289]$ (for $n=2, \sigma=2$,) with iterations  $M=5,~10$, respectively. Since exact solution is not known so we define the absolute residual error as
 \begin{align*}
 E_{res }^{M}(x)=\big| (x^2\phi'_{M}(x))'-\sigma^2 x^2 \phi^{n}_{M}(x) \big|,~~e_{res }^{M}(x)=\big| (x^2\psi'_{M}(x))'- \sigma^2 x^2 \psi^{n}_{M}(x) \big|.
\end{align*}
The numerical results of the absolute  residual errors and the approximate solutions are shown in Tables \ref{tab2}, \ref{tab3} and  \ref{tab4}. One can observe that  the residual error not converging to zero with the increase in $\sigma$ and $n$ by ADMGF technique whereas the proposed method OHAM gives stable solution and converges to exact solution.

\begin{table}[htbp]
\caption{Numerical results of residual error and solutions  of Problem \ref{prob2} for  $n=1.5,~ \sigma=1$}\label{tab2}
\centering
\vspace{-0.3cm}
\renewcommand{\arraystretch}{1.1}
\setlength{\tabcolsep}{0.14in}
\begin{tabular}{l|cc| cc|cc}
\hline
\cline{1-7}
$x$ & $e_{res}^{5}$  & $E_{res}^{5}$ & $e_{res}^{10}$    &   $E_{res}^{10}$    &   $\psi_{10}$ & $\phi_{10}$\\
\cline{1-7}
0.1	&	5.20E-04	&	7.56E-06	&	3.75E-07	&	1.22E-13	&	0.859202	&	0.859202	\\
0.2	&	4.88E-04	&	7.15E-06	&	3.50E-07	&	4.96E-12	&	0.863188	&	0.863188	\\
0.3	&	4.38E-04	&	6.52E-06	&	3.10E-07	&	1.34E-11	&	0.869870    &	0.869870    \\
0.4	&	3.74E-04	&	5.72E-06	&	2.61E-07	&	2.56E-11	&	0.879303	&	0.879303	\\
0.5	&	3.02E-04	&	4.79E-06	&	2.07E-07	&	4.03E-11	&	0.891566	&	0.891566	\\
0.6	&	2.27E-04	&	3.66E-06	&	1.53E-07	&	5.18E-11	&	0.906766	&	0.906766	\\
0.7	&	1.56E-04	&	2.08E-06	&	1.03E-07	&	4.09E-11	&	0.925033	&	0.925033	\\
0.8	&	9.25E-05	&	5.69E-07	&	5.95E-08	&	4.75E-11	&	0.946527	&	0.946527	\\
0.9	&	3.98E-05	&	5.62E-06	&	2.50E-08	&	3.58E-10	&	0.971441	&	0.971441	\\
\hline
\end{tabular}
\end{table}
\begin{table}[htbp]
\caption{Results of residual error and solutions of Problem \ref{prob2} for $n=2, \sigma=1.5$ }\label{tab3}
\centering
\vspace{-0.3cm}
\renewcommand{\arraystretch}{1.0}
\setlength{\tabcolsep}{0.14in}
\begin{tabular}{l|cc| cc|cc}
\hline
\cline{1-7}
$x$ & $e_{res}^{5}$   & $E_{res}^{5}$ & $e_{res}^{10}$     &   $E_{res}^{10}$    &   $\psi_{10}$ & $\phi_{10}$\\
\cline{1-7}
0.1	&	3.69E-01	&	1.21E-03	&	1.34E-01	&	9.57E-08	&	0.759370	 &	0.750609	\\
0.2	&	3.46E-01	&	1.14E-03	&	1.25E-01	&	1.88E-07	&	0.765211	&	0.756965	\\
0.3	&	3.10E-01	&	1.03E-03	&	1.10E-01	&	3.33E-07	&	0.775144	&	0.767704	\\
0.4	&	2.64E-01	&	8.79E-04	&	9.13E-02	&	5.10E-07	&	0.789464	&	0.783048	\\
0.5	&	2.13E-01	&	6.84E-04	&	7.16E-02	&	6.56E-07	&	0.808584	&	0.803324	\\
0.6	&	1.60E-01	&	4.01E-04	&	5.24E-02	&	5.93E-07	&	0.833035	&	0.828980	\\
0.7	&	1.10E-01	&	7.74E-05	&	3.50E-02	&	1.88E-07	&	0.863480	&	0.860603	\\
0.8	&	6.53E-02	&	9.87E-04	&	2.03E-02	&	3.12E-06	&	0.900738	&	0.898953	\\
0.9	&	2.85E-02	&	2.82E-03	&	8.71E-03	&	1.21E-05	&	0.945823	&	0.945003	\\
\hline
\end{tabular}
\end{table}

\begin{table}[htbp]
\caption{Results of residual error and solutions of Problem \ref{prob2} for $n=2, \sigma=2$ }\label{tab4}
\centering
\vspace{-0.3cm}
\renewcommand{\arraystretch}{1.0}
\setlength{\tabcolsep}{0.15in}
\begin{tabular}{l|cc| cc|cc}
\hline
\cline{1-7}
$x$ & $e_{res}^{5}$   & $E_{res}^{5}$ & $e_{res}^{10}$     &   $E_{res}^{10}$    &   $\psi_{10}$ & $\phi_{10}$\\
\cline{1-7}
0.1	&	0.083	&	1.63E-04	&	1.140	&	9.69E-07	&	4.265570	&	0.641604	\\
0.2	&	0.316	&	6.14E-04	&	4.160	&	4.10E-06	&	4.060708	&	0.649873	\\
0.3	&	0.650	&	1.23E-03	&	7.990	&	9.92E-06	&	3.741832	&	0.663942	\\
0.4	&	1.013	&	1.87E-03	&	11.37	&	1.85E-05	&	3.339042	&	0.684262	\\
0.5	&	1.317	&	2.30E-03	&	13.26	&	2.67E-05	&	2.887879	&	0.711509	\\
0.6	&	1.482	&	2.17E-03	&	13.20	&	1.83E-05	&	2.424517	&	0.746636	\\
0.7	&	1.444   &	6.53E-04	&	11.30	&	6.65E-05	&	1.981583	&	0.790944	\\
0.8	&	1.176	&	4.45E-03	&	8.071	&	4.25E-04	&	1.585305	&	0.846200	\\
0.9	&	0.684	&	1.89E-02	&	4.130	&	1.67E-04	&	1.254310	&	0.914796\\
\hline
\end{tabular}
\end{table}

\begin{problem}\label{prob3}
\rm{\textbf{Distribution of Heat Sources in the Human Head} \cite{flesch1975distribution,gray1980distribution,duggan1986pointwise}}
\end{problem}
Consider singular boundary value problems \cite{singh2014efficient}
\begin{eqnarray*}
\left.
  \begin{array}{ll}
  \displaystyle -(x^2y'(x))' =\delta \;x^2 e^{-y(x)}~~~~0<x<1,\vspace{0.2cm }\\
  \displaystyle y'(0)=0,~~~\; \alpha_2y(1)+\beta_2 y'(1)=\gamma_2,\\
 \end{array}
\right\}
\end{eqnarray*}
Here,  $p(x)=q(x)=x^{2}$, $f(x,y)=\delta \;e^{-y(x)}$  and  $\delta=1$. Applying the OHAM  \eqref{sec2:eq27}
with an initial guess $y_0(x)=0$,  we get the approximate solution  $\phi_M(x,c_0).$
Using the formula \eqref{sec2:eq30},  we obtain optimal values, $\hat{c}_0=[-0.6842013;-0.6666463]$  (for $\alpha_2=\beta_2=1,\gamma_2=0$) and  $\hat{c}_0=[-0.7759493;-0.7701234]$ (for $\alpha_2=2,\beta_2=1,\gamma_2=0$)  with iterations  $M=5$ and $M=10$, respectively.   Since exact solution is not known so we define the absolute residual errors as
\begin{align*}
 E_{res }^{M}(x)=\big| (x^2\phi'_{M}(x))'+\delta x^2 e^{-\phi_{M}(x)} \big|,~~e_{res }^{M}(x)=\big| (x^2\psi'_{M}(x))'+ \delta x^2 e^{-\psi_{M}(x)}\big|.
\end{align*}
The numerical results are given in Tables \ref{tab5} and  \ref{tab6}.  We observe that  that the residual error not converging to zero by ADMGF method whereas the OHAM gives stable solution and converges to exact solution.
\begin{table}[htbp]
\caption{Results of residual error and solutions of Problem \ref{prob3} when  $\alpha_2=\beta_2=1,\gamma_2=0$}\label{tab5}
\centering
\vspace{-0.3cm}
\renewcommand{\arraystretch}{1.0}
\setlength{\tabcolsep}{0.13in}
\begin{tabular}{l|cc| cc|cc}
\hline
\cline{1-7}
$x$ & $e_{res}^{5}$   & $E_{res}^{5}$ & $e_{res}^{10}$    &   $E_{res}^{10}$    &   $\psi_{10}$ & $\phi_{10}$\\
\cline{1-7}
0.1	&	1.18E-01	&	2.05E-04	&	8.05E-02	&	2.41E-06	&	0.3442719	&	0.3663613	\\
0.2	&	1.15E-01	&	1.87E-04	&	7.87E-02	&	2.40E-06	&	0.3411278	&	0.3628931	\\
0.3	&	1.12E-01	&	1.57E-04	&	7.58E-02	&	2.38E-06	&	0.3358608	&	0.3570965	\\
0.4	&	1.07E-01	&	1.15E-04	&	7.19E-02	&	2.37E-06	&	0.3284310	&	0.3489474	\\
0.5	&	1.01E-01	&	6.13E-05	&	6.73E-02	&	2.36E-06	&	0.3187835	&	0.3384112	\\
0.6	&	9.44E-02	&	4.61E-06	&	6.20E-02	&	2.38E-06	&	0.3068490	&	0.3254426	\\
0.7	&	8.68E-02	&	8.39E-05	&	5.63E-02	&	2.44E-06	&	0.2925442	&	0.3099852	\\
0.8	&	7.87E-02	&	1.79E-04	&	5.05E-02	&	2.57E-06	&	0.2757729	&	0.2919703	\\
0.9	&	7.04E-02	&	2.94E-04	&	4.46E-02	&	2.79E-06	&	0.2564260	&	0.2713162	\\
\hline
\end{tabular}
\end{table}

\begin{table}[htbp]
\caption{Results of residual error and solutions of Problem \ref{prob3}  when  $\alpha_2=2,\beta_2=1,\gamma_2=0$}\label{tab6}
\centering
\renewcommand{\arraystretch}{1.0}
\setlength{\tabcolsep}{0.14in}
\begin{tabular}{l|cc| cc|cc}
\hline
\cline{1-7}
$x$ & $e_{res}^{5}$  & $E_{res}^{5}$& $e_{res}^{10}$     &   $E_{res}^{10}$   &   $\psi_{10}$ & $\phi_{10}$\\
\cline{1-7}
0.1	&	1.35E-02	&	6.11E-05	&	8.11E-04	&	9.12E-08	&	0.2686241	&	0.2687568	\\
0.2	&	1.31E-02	&	5.54E-05	&	7.77E-04	&	8.83E-08	&	0.2648035	&	0.2649327	\\
0.3	&	1.24E-02	&	4.62E-05	&	7.25E-04	&	8.39E-08	&	0.2584162	&	0.2585397	\\
0.4	&	1.14E-02	&	3.37E-05	&	6.58E-04	&	7.87E-08	&	0.2494321	&	0.2495481	\\
0.5	&	1.03E-02	&	1.82E-05	&	5.80E-04	&	7.35E-08	&	0.2378088	&	0.2379158	\\
0.6	&	9.08E-03	&	2.51E-07	&	4.98E-04	&	6.93E-08	&	0.2234908	&	0.2235876	\\
0.7	&	7.77E-03	&	2.23E-05	&	4.16E-04	&	6.72E-08	&	0.2064084	&	0.2064944	\\
0.8	&	6.46E-03	&	4.92E-05	&	3.38E-04	&	6.87E-08	&	0.1864771	&	0.1865520	\\
0.9	&	5.21E-03	&	8.36E-05	&	2.68E-04	&	7.65E-08	&	0.1635958	&	0.1636596	\\
\hline
\end{tabular}
\end{table}

\begin{problem}\label{prob4}
\rm{\textbf{Oxygen Diffusion in a Spherical Cell \cite{lin1976oxygen,mcelwain1978re,anderson1980complementary}}}
\end{problem}
Consider the following nonlinear singular boundary value problem:
\begin{eqnarray*}
\left.
  \begin{array}{ll}
  \displaystyle (x^2y'(x))' =n  x^2 \frac{y(x)}{y(x)+k},~~~0<x<1\vspace{0.2cm }\\
  \displaystyle y'(0)=0,~~~~5 y(1)+  y'(1)=5.
\end{array}
\right\}
\end{eqnarray*}
Here,  $p(x)=q(x)=x^{2}$, $f(x,y)=n\frac{y(x)}{y(x)+k}$ with  $n=0.76129$, $k=0.03119$, $\alpha_2=\gamma_2=5$ and $\beta_2=1$ as in \cite{wazwaz2011variational,singh2014efficient}. Applying the OHAM  \eqref{sec2:eq27} with an initial guess $u_0(x)=1$, we get the approximation to solution $\phi_M(x,c_0).$
Using the formula  \eqref{sec2:eq30}, we obtain optimal values $\hat{c}_0=[-1.045949; -1.010201]$ with  $M=5,10$, respectively. 
Since exact solution is not known so we define the absolute residual error as
 \begin{align*}
 E_{res }^{M}(x)=\bigg| (x^2\phi'_{M}(x))'-x^2 n  \frac{\phi_{M}(x)}{\phi_{M}(x)+k}\bigg|,~~e_{res }^{M}(x)=\bigg| (x^2\psi'_{M}(x))'-x^2 n  \frac{\psi_{M}(x)}{\psi_{M}(x)+k} \bigg|.
\end{align*}
The numerical results ate presented in   Table \ref{tab7}. From the numerical results we observe that OHAM give slightly better results compared to ADMGF method.

\begin{table}[htbp]
\caption{Numerical results of residual error and solutions of Problem \ref{prob4}}\label{tab7}
\centering
\vspace{-0.3cm}
\renewcommand{\arraystretch}{1.2}
\setlength{\tabcolsep}{0.13in}
\begin{tabular}{l|cc| cc|cc}
\hline
\cline{1-7}
$x$ & $e_{res}^{5}$  & $E_{res}^{5}$ & $e_{res}^{10}$     &   $E_{res}^{10}$    &   $\psi_{10}$ & $\phi_{10}$\\
\cline{1-7}
0.1	&	2.80E-06	&	7.95E-07	&	1.95E-10	&	1.04E-10	&	0.829706092	&	0.829706092	\\
0.2	&	2.49E-06	&	6.94E-07	&	1.50E-10	&	7.76E-11	&	0.833374734	&	0.833374734	\\
0.3	&	2.03E-06	&	5.54E-07	&	9.42E-11	&	4.61E-11	&	0.839489914	&	0.839489914	\\
0.4	&	1.50E-06	&	4.07E-07	&	4.56E-11	&	2.00E-11	&	0.848052785	&	0.848052785	\\
0.5	&	9.93E-07	&	2.83E-07	&	1.45E-11	&	4.59E-12	&	0.859064927	&	0.859064927	\\
0.6	&	5.66E-07	&	2.01E-07	&	6.60E-13	&	1.23E-12	&	0.872528320	&	0.872528320	\\
0.7	&	2.63E-07	&	1.59E-07	&	2.52E-12	&	1.82E-12	&	0.888445306	&	0.888445306	\\
0.8	&	8.70E-08	&	1.49E-07	&	1.77E-12	&	9.98E-13	&	0.906818548	&	0.906818548	\\
0.9	&	1.13E-08	&	1.52E-07	&	7.45E-13	&	4.01E-13	&	0.927650988	&	0.927650988	\\
\hline
\end{tabular}
\end{table}

\begin{problem}\label{prob5}
\rm{\textbf{Perturbed Second Kind Lane-Emden  Equation \cite{reger2013lane} }}
\end{problem}
Consider the following perturbed singular boundary value problem:
\begin{eqnarray*}
\left.
  \begin{array}{ll}
 \displaystyle -(x^{\alpha}y'(x))' =\delta \;x^{\alpha} \displaystyle  \exp\bigg({\frac{y(x)}{1+\epsilon \; y(x)}}\bigg),~~0<x<1 \vspace{0.2cm }\\
 \displaystyle y'(0)=0,~~~2\;y(1)+y'(1)=0.
\end{array}
\right\}
 \end{eqnarray*}
Here,  $p(x)=q(x)=x^{\alpha}$, $f(x,y)=\delta  \exp\big({\frac{y(x)}{1+\epsilon \; y(x)}}\big)$  with $\alpha_2=2$,  $\beta_2=1$ and  $\gamma_2=0$  as in \cite{singh2017optimal}.
Applying the OHAM  \eqref{sec2:eq27} with an initial guess $u_0(x)=0$ , we get the approximation $\phi_M(x,c_0).$ Using the formula defined by \eqref{sec2:eq30}, we obtain optimal values   $\hat{c}_0=[-0.432512;-0.381111;-0.284943]$ (for $\alpha=1$) and   $\hat{c}_0$=$[-0.608235;-0.471209;-0.381567]$ (for $\alpha=2$) with ($\epsilon=5,10,15$), $M=10$, respectively. 
Since exact solution is not known so we define the absolute residual errors as
 \begin{align*}
 &E_{res }^{M}(x)=\bigg| (x^{\alpha}\phi'_{M}(x))'+x^{\alpha} \delta \exp\bigg(\frac{\phi_{M}(x)}{1+\epsilon \phi_{M}(x)}\bigg)\bigg|,\\
&e_{res }^{M}(x)=\bigg| (x^{\alpha}\psi'_{M}(x))'+x^{\alpha} \delta \exp\bigg(\frac{\psi_{M}(x)}{1+\epsilon \psi_{M}(x)}\bigg)\bigg|.
\end{align*}
In Tables \ref{tab8} and \ref{tab9},   we consider the influence of $\epsilon$ on the  residual error for ($\epsilon=5,10,15$) with  $\alpha=2$ and $M=10$. In each cases, we observe that the residual error not converging to zero with an increases in $\epsilon$ by ADMGF technique whereas the  OHAM gives stable solution and converges to exact solution.

\begin{table}[htbp]
\caption{Numerical results of absolute residual error Problem \ref{prob5} when  $\alpha=1,~\delta=1$}\label{tab8}
\centering
\vspace{-0.3cm}
\renewcommand{\arraystretch}{1.0}
\setlength{\tabcolsep}{0.14in}
\begin{tabular}{l|cc| cc|cc}
\hline
&&$\epsilon=5$& &$\epsilon=10$& &$\epsilon=15$ \\
\cline{1-7}
$x$ & $e_{res}^{10}$ &  $E_{res}^{10}$ &    $e_{res}^{10}$    &   $E_{res}^{10}$   &  $e_{res}^{10}$ & $E_{res}^{10}$\\
\cline{1-7}
0.1	&	96.820  &	3.79E-04	&	1190.90	&	3.95E-05	&	173751.94	&	1.73E-03	\\
0.2	&	200.650	&	4.41E-04	&	2334.05	&	2.93E-04	&	299967.18	&	1.11E-03	\\
0.3	&	312.307	&	2.20E-05	&	3295.97	&	7.10E-04	&	350180.75	&	2.71E-03	\\
0.4	&	423.542	&	1.03E-03	&	3858.23	&	8.85E-04	&	324708.56	&	8.77E-03	\\
0.5	&	519.051	&	2.45E-03	&	3823.73	&	2.47E-04	&	248739.93	&	1.52E-02	\\
0.6	&	582.390	&	4.01E-03	&	3159.91	&	1.51E-03	&	157883.64	&	2.08E-02	\\
0.7	&	602.770	&	5.44E-03	&	2068.78	&	4.10E-03	&	081587.42	&	2.52E-02	\\
0.8	&	579.289	&	6.62E-03	&	0914.41	&	6.73E-03	&	033136.89	&	2.90E-02	\\
0.9	&	520.645	&	7.53E-03	&	0040.79	&	8.66E-03	&	010098.14	&	3.26E-02	\\

\hline
\end{tabular}
\end{table}

\begin{table}[htbp]
\caption{Numerical results of absolute residual error Problem \ref{prob5} when  $\alpha=2,~\delta=1$}\label{tab9}
\centering
\vspace{-0.3cm}
\renewcommand{\arraystretch}{1.0}
\setlength{\tabcolsep}{0.14in}
\begin{tabular}{l|cc| cc|cc}
\hline
&&$\epsilon=5$& &$\epsilon=10$& &$\epsilon=15$ \\
\cline{1-7}
$x$ & $e_{res}^{10}$ &  $E_{res}^{10}$ &    $e_{res}^{10}$    &   $E_{res}^{10}$   &  $e_{res}^{10}$ & $E_{res}^{10}$\\
\cline{1-7}
0.1	&	380.762	&	2.28E-03		&	1279.93	&	1.06E-03	&	29880.52	&	4.02E-03	\\
0.2	&	332.801	&	1.83E-03		&	1129.84	&	4.31E-04	&	25360.17	&	6.66E-04	\\
0.3	&	264.086	&	1.09E-03		&	0917.57	&	3.92E-04	&	19118.32	&	3.28E-03	\\
0.4	&	187.844	&	1.19E-04		&	0686.12	&	1.18E-03	&	12594.80	&	6.41E-03	\\
0.5	&	116.761	&	9.99E-04		&	0474.09	&	1.77E-03	&	07033.77	&	8.11E-03	\\
0.6	&	059.804	&	2.13E-03		&	0305.58	&	2.12E-03	&	03128.11	&	8.61E-03	\\
0.7	&	020.680	&	3.14E-03		&	0187.06	&	2.24E-03	&	00928.97	&	8.56E-03	\\
0.8	&	001.774	&	3.97E-03		&	0111.44	&	2.21E-03	&	00018.88	&	8.44E-03	\\
0.9	&	011.741	&	4.57E-03		&	0065.80	&	2.08E-03	&	00178.08	&	8.39E-03	\\
\hline
\end{tabular}
\end{table}

\section{Conclusions}
In this paper, we have examined the  doubly singular  boundary value problems  with Dirichlet/Neumann boundary conditions at $x=0$ and Robin type boundary conditions at $x=1$,   arising in  the reaction-diffusion process in a porous spherical catalyst \cite{fogler1999elements}, oxygen diffusion in a spherical cell \cite{lin1976oxygen}, heat sources in the human head \cite{duggan1986pointwise} and the perturbed second kind Lane-Emden  equation is used in modelling a thermal explosion  \cite{reger2013lane}. Due to the presence of singularity at $x=0$ as well as discontinuity of $q(x)$ at $x=0$, these problems pose difficulties in obtaining their solutions. In this paper, a new formulation of the singular  boundary value problems has been presented. To overcome the singular behavior at the origin, with the help of Green's function theory the problem  has been transformed into an equivalent Fredholm integral equation. Then  the optimal homotopy analysis method is applied to solve integral form of problem. The optimal control-convergence parameter involved in the components of the series solution has been obtained by minimizing the squared residual error equation. For speed up the calculations, the discrete averaged residual error has been used to obtain optimal value of the  adjustable parameter $c_0$ to control the convergence of solution. Numerical results obtained by OHAM are better than the results obtained by the ADMGF \cite{singh2014efficient} and are in good agreement with exact solutions, as shown in Tables \ref{tab1}-\ref{tab9}. Unlike ADMGF \cite{singh2014efficient}, the  OHAM  always gives fast convergent series solution as shown in Tables. Convergence analysis and  error estimate of the proposed method have been discussed.
The proposed method has successfully applied to the  perturbed second kind Lane-Emden  Equation \cite{reger2013lane} whereas other method fails to give covergenct series solution as shown in  Tables \ref{tab8} and \ref{tab9}.


\begin{thebibliography}{10}
\expandafter\ifx\csname url\endcsname\relax
  \def\url#1{\texttt{#1}}\fi
\expandafter\ifx\csname urlprefix\endcsname\relax\def\urlprefix{URL }\fi
\expandafter\ifx\csname href\endcsname\relax
  \def\href#1#2{#2} \def\path#1{#1}\fi

\bibitem{Bobisud1990}
L.~Bobisud, Existence of solutions for nonlinear singular boundary value problems, Applicable Analysis 35~(1-4) (1990) 43--57.

\bibitem{singh2014approximate}
R.~Singh, J.~Kumar, G.~Nelakanti, Approximate series solution of singular
  boundary value problems with derivative dependence using Green's function
  technique, Computational and Applied Mathematics 33~(2) (2014) 451--467.

\bibitem{singh2014adomian}
R.~Singh, J.~Kumar, The Adomian decomposition method with Green's function for
  solving nonlinear singular boundary value problems, Journal of Applied
  Mathematics and Computing 44~(1-2) (2014) 397--416.

\bibitem{thomas1927calculation}
L.~Thomas, The calculation of atomic fields, in: Mathematical Proceedings of
  the Cambridge Philosophical Society, Vol.~23, Cambridge Univ Press, 1927, pp.
  542--548.

\bibitem{fermi1927metodo}
E.~Fermi, Un metodo statistico per la determinazione di alcune priorieta
  dell'atome, Rend. Accad. Naz. Lincei 6~(602-607) (1927) 32.

\bibitem{wazwaz2017solving}
A.M. Wazwaz, Solving the non-isothermal reaction-diffusion model equations in
  a spherical catalyst by the variational iteration method, Chemical Physics
  Letters 679 (2017) 132--136.

\bibitem{lin1976oxygen}
S.~Lin, Oxygen diffusion in a spherical cell with nonlinear oxygen uptake
  kinetics, Journal of Theoretical Biology 60~(2) (1976) 449--457.

\bibitem{anderson1980complementary}
N.~Anderson, A.~Arthurs, Complementary variational principles for diffusion
  problems with {Michaelis-Menten kinetics}, Bulletin of Mathematical Biology
  42~(1) (1980) 131--135.

\bibitem{flesch1975distribution}
U.~Flesch, The distribution of heat sources in the human head: a theoretical
  consideration, Journal of Theoretical Biology 54~(2) (1975) 285--287.

\bibitem{gray1980distribution}
B.~Gray, The distribution of heat sources in the human head-theoretical
  considerations, Journal of Theoretical Biology 82~(3) (1980) 473--476.

\bibitem{duggan1986pointwise}
R.~Duggan, A.~Goodman, Pointwise bounds for a nonlinear heat conduction model
  of the human head, Bulletin of Mathematical Biology 48~(2) (1986) 229--236.

\bibitem{chawla1987existence}
M.~Chawla, P.~Shivakumar, On the existence of solutions of a class of singular
  nonlinear two-point boundary value problems, Journal of Computational and
  Applied Mathematics 19~(3) (1987) 379--388.

\bibitem{dunninger1986existence}
D.~Dunninger, J.~Kurtz, Existence of solutions for some nonlinear singular
  boundary value problem, Journal of Mathematical Analysis and Applications
  115~(2) (1986) 396--405.

\bibitem{pandey2009note}
R.~Pandey, A.~K. Verma, A note on existence-uniqueness results for a class of
  doubly singular boundary value problems, Nonlinear Analysis: Theory, Methods
  \& Applications 71~(7) (2009) 3477--3487.

\bibitem{Reddien1973projection}
G.~Reddien, Projection methods and singular two point boundary value problems,
  Numerische Mathematik 21~(3) (1973) 193--205.

\bibitem{russell1975numerical}
R.~Russell, L.~Shampine, Numerical methods for singular boundary value
  problems, SIAM Journal on Numerical Analysis 12~(1) (1975) 13--36.

\bibitem{jamet1970convergence}
P.~Jamet, On the convergence of finite-difference approximations to
  one-dimensional singular boundary-value problems, Numerische Mathematik
  14~(4) (1970) 355--378.

\bibitem{chawla1982finite}
M.~Chawla, C.~Katti, Finite difference methods and their convergence for a
  class of singular two point boundary value problems, Numerische Mathematik
  39~(3) (1982) 341--350.

\bibitem{chawla1984finite}
M.~Chawla, C.~Katti, A finite-difference method for a class of singular
  two-point boundary-value problems, IMA Journal of Numerical Analysis 4~(4)
  (1984) 457--466.

\bibitem{iyengar1986spline}
S.~Iyengar, P.~Jain, Spline finite difference methods for singular two point
  boundary value problems, Numerische Mathematik 50~(3) (1986) 363--376.

\bibitem{kadalbajoo2007b}
M.~K. Kadalbajoo, V.~Kumar, B-spline method for a class of singular two-point
  boundary value problems using optimal grid, Applied mathematics and
  computation 188~(2) (2007) 1856--1869.

\bibitem{kumar2007higher}
M.~Kumar, Higher order method for singular boundary-value problems by using
  spline function, Applied Mathematics and Computation 192~(1) (2007) 175--179.

\bibitem{kanth2003numerical}
A.~R. Kanth, Y.~Reddy, A numerical method for singular two point boundary value
  problems via chebyshev economizition, Applied Mathematics and Computation
  146~(2) (2003) 691--700.

\bibitem{kanth2005cubic}
A.~R. Kanth, Y.~Reddy, Cubic spline for a class of singular two-point boundary
  value problems, Applied Mathematics and Computation 170~(2) (2005) 733--740.

\bibitem{kanth2006cubic}
A.~Ravi~Kanth, V.~Bhattacharya, Cubic spline for a class of non-linear singular
  boundary value problems arising in physiology, Applied Mathematics and
  Computation 174~(1) (2006) 768--774.

\bibitem{kanth2007cubic}
A.~R. Kanth, Cubic spline polynomial for non-linear singular two-point boundary
  value problems, Applied mathematics and computation 189~(2) (2007)
  2017--2022.

\bibitem{inc2005different}
M.~Inc, M.~Ergut, Y.~Cherruault, A different approach for solving singular
  two-point boundary value problems, Kybernetes: The International Journal of
  Systems \& Cybernetics 34~(7) (2005) 934--940.

\bibitem{mittal2008solution}
R.~Mittal, R.~Nigam, Solution of a class of singular boundary value problems,
  Numerical Algorithms 47~(2) (2008) 169--179.

\bibitem{khuri2010novel}
S.~Khuri, A.~Sayfy, A novel approach for the solution of a class of singular
  boundary value problems arising in physiology, Mathematical and Computer
  Modelling 52~(3) (2010) 626--636.

\bibitem{ebaid2011new}
A.~Ebaid, A new analytical and numerical treatment for singular two-point
  boundary value problems via the {Adomian} decomposition method, Journal of
  Computational and Applied Mathematics 235~(8) (2011) 1914--1924.

\bibitem{Kumar2010}
M.~Kumar, N.~Singh, Modified {Adomian} decomposition method and computer
  implementation for solving singular boundary value problems arising in
  various physical problems, Computers \& Chemical Engineering 34~(11) (2010)
  1750--1760.

\bibitem{singh2013numerical}
R.~Singh, J.~Kumar, G.~Nelakanti, Numerical solution of singular boundary value
  problems using Green's function and improved decomposition method, Journal of
  Applied Mathematics and Computing 43~(1-2) (2013) 409--425.

\bibitem{singh2014efficient}
R.~Singh, J.~Kumar, An efficient numerical technique for the solution of
  nonlinear singular boundary value problems, Computer Physics Communications
  185~(4) (2014) 1282--1289.

\bibitem{singh2016efficient}
R.~Singh, A.-M. Wazwaz, J.~Kumar, An efficient semi-numerical technique for
  solving nonlinear singular boundary value problems arising in various
  physical models, International Journal of Computer Mathematics 93~(8) (2016)
  1330--1346.

\bibitem{wazwaz2011comparison}
A.~Wazwaz, R.~Rach, Comparison of the {Adomian} decomposition method and the
  variational iteration method for solving the lane-emden equations of the
  first and second kinds, Kybernetes 40~(9/10) (2011) 1305--1318.

\bibitem{wazwaz2011variational}
A.~Wazwaz, The variational iteration method for solving nonlinear singular
  boundary value problems arising in various physical models, Communications in
  Nonlinear Science and Numerical Simulation 16~(10) (2011) 3881--3886.

\bibitem{ravi2010he}
A.~Ravi~Kanth, K.~Aruna, He's variational iteration method for treating
  nonlinear singular boundary value problems, Computers \& Mathematics with
  Applications 60~(3) (2010) 821--829.

\bibitem{singh2017optimal}
R.~Singh, N.~Das, J.~Kumar, The optimal modified variational iteration method
  for the Lane-Emden equations with Neumann and Robin boundary conditions, The
  European Physical Journal Plus 132~(6) (2017) 251.

\bibitem{danish2012note}
M.~Danish, S.~Kumar, S.~Kumar, A note on the solution of singular boundary
  value problems arising in engineering and applied sciences: Use of {OHAM},
  Computers \& Chemical Engineering 36 (2012) 57--67.

\bibitem{roul2016new}
P.~Roul, U.~Warbhe, New approach for solving a class of singular boundary value
  problem arising in various physical models, Journal of Mathematical Chemistry
  54~(6) (2016) 1255--1285.

\bibitem{liao1995approximate}
S.-J. Liao, An approximate solution technique not depending on small
  parameters: a special example, International Journal of Non-Linear Mechanics
  30~(3) (1995) 371--380.

\bibitem{liao2003beyond}
S.~Liao, Beyond perturbation: introduction to the homotopy analysis method, CRC
  press, 2003.

\bibitem{liao2007general}
S.~Liao, Y.~Tan, A general approach to obtain series solutions of nonlinear
  differential equations, Studies in Applied Mathematics 119~(4) (2007)
  297--354.

\bibitem{liao2009series}
S.~Liao, Series solution of nonlinear eigenvalue problems by means of the
  homotopy analysis method, Nonlinear Analysis: Real World Applications 10~(4)
  (2009) 2455--2470.

\bibitem{abbasbandy2013determination}
S Abbasbandy, M Jalili, Determination of optimal convergence-control parameter value in homotopy analysis method, Numerical Algorithms 64~(4) (2013) 593--605.


\bibitem{fan2013optimal}
Fan, Tao and You, Xiangcheng, Optimal homotopy analysis method for nonlinear differential equations in the boundary layer, Numerical Algorithms 62~(2) (2013) 337–354.


\bibitem{marinca2008application}
V.~Marinca, N.~Heri{\c{s}}anu, Application of optimal homotopy asymptotic
  method for solving nonlinear equations arising in heat transfer,
  International Communications in Heat and Mass Transfer 35~(6) (2008)
  710--715.

\bibitem{hericsanu2010accurate}
N.~Heri{\c{s}}anu, V.~Marinca, Accurate analytical solutions to oscillators
  with discontinuities and fractional-power restoring force by means of the
  optimal homotopy asymptotic method, Computers \& Mathematics with
  Applications 60~(6) (2010) 1607--1615.

\bibitem{hashmi2012numerical}
M.~Hashmi, N.~Khan, S.~Iqbal, Numerical solutions of weakly singular volterra
  integral equations using the optimal homotopy asymptotic method, Computers \&
  Mathematics with Applications 64~(6) (2012) 1567--1574.

\bibitem{liao2012homotopy}
S.~Liao, Homotopy analysis method in nonlinear differential equations,
  Springer, 2012.

\bibitem{odibat2010study}
Z.~M. Odibat, A study on the convergence of homotopy analysis method, Applied
  Mathematics and Computation 217~(2) (2010) 782--789.

\bibitem{hetmaniok2014usage}
E.~Hetmaniok, D.~S{\l}ota, T.~Trawi{\'n}ski, R.~Witu{\l}a, Usage of the
  homotopy analysis method for solving the nonlinear and linear integral
  equations of the second kind, Numerical Algorithms 67~(1) (2014) 163--185.

\bibitem{liao2010optimal}
S.~Liao, An optimal homotopy-analysis approach for strongly nonlinear
  differential equations, Communications in Nonlinear Science and Numerical
  Simulation 15~(8) (2010) 2003--2016.

\bibitem{cherruault1989convergence}
Y.~Cherruault, Convergence of {Adomian}'s method, Kybernetes 18~(2) (1989)
  31--38.

\bibitem{fogler1999elements}
H.~S. Fogler, et~al., Elements of chemical reaction engineering.

\bibitem{mcelwain1978re}
D.~McElwain, A re-examination of oxygen diffusion in a spherical cell with
  {Michaelis-Menten} oxygen uptake {kinetics}, Journal of Theoretical Biology
  71 (1978) 255--263.

\bibitem{reger2013lane}
K.~Reger, R.~Van~Gorder, Lane-Emden equations of second kind modelling thermal
  explosion in infinite cylinder and sphere, Applied Mathematics and Mechanics
  34~(12) (2013) 1439--1452.




\end{thebibliography}

\end{document}